\documentclass[11pt]{amsart}

\usepackage{epigamath}

\usepackage[english]{babel}

\numberwithin{equation}{section}

\usepackage[shortlabels]{enumitem}

\usepackage{amsmath,amsfonts,amssymb,amscd,mathtools,bm}
\usepackage{tikz}
\usepackage{tikz-cd} 
\usepackage{multirow}
\usepackage{caption}
\usepackage{subcaption}
\usepackage{cleveref}

\newtheorem{thm}{Theorem}[section]
\newtheorem{cor}[thm]{Corollary}
\newtheorem{lemma}[thm]{Lemma}
\newtheorem{prop}[thm]{Proposition}

\theoremstyle{definition}
\newtheorem{defn}[thm]{Definition}

\theoremstyle{remark}
\newtheorem{example}[thm]{Example}
\newtheorem{remark}[thm]{Remark}

\setlist[enumerate,1]{label={\rm(\arabic*)}, ref={\rm\arabic*}} 

\newcommand{\lra}{\longrightarrow}
\newcommand{\longhookrightarrow}{\lhook\joinrel\longrightarrow}
\newcommand{\supth}[1]{\ensuremath{#1^{\mathrm{th}}}}

\DeclareMathOperator{\ev}{ev}
\DeclareMathOperator{\Fl}{Fl}
\DeclareMathOperator{\Gr}{Gr}
\DeclareMathOperator{\QH}{QH}
\DeclareMathOperator{\QK}{QK}
\DeclareMathOperator{\Jac}{Jac}
\DeclareMathOperator{\HH}{H}
\DeclareMathOperator{\K}{K}
\DeclareMathOperator{\pt}{pt}
\DeclareMathOperator{\Rep}{Rep}
\DeclareMathOperator{\Ker}{Ker}
\DeclareMathOperator{\Sp}{Sp}

\newcommand{\C}{{\mathbb C}}
\newcommand{\bP}{{\mathbb P}}
\newcommand{\Z}{{\mathbb Z}}

\newcommand{\cF}{{\mathcal F}}
\newcommand{\cO}{{\mathcal O}}
\newcommand{\cQ}{{\mathcal Q}}
\newcommand{\cS}{{\mathcal S}}

\newcommand{\wb}{\overline}
\newcommand{\Mb}{\wb{\mathcal M}}

\newcommand{\M}{\overline{M}}

\DeclarePairedDelimiter{\angles}{\langle}{\rangle}
\newcommand{\egw}[2]{\angles{ #1 }^T_{#2}}

\def\poly{{\mathrm{poly}}}

\EpigaVolumeYear{9}{2025} \EpigaArticleNr{25} \ReceivedOn{November 18, 2024}
\InFinalFormOn{June 17, 2025}
\AcceptedOn{July 9, 2025}

\title{A Nakayama result for the quantum K-theory\\ of homogeneous spaces}
\titlemark{QK Nakayama}

\author{Wei Gu}
\address{Zhejiang Institute of Modern Physics, School of Physics, Zhejiang University, Hangzhou, Zhejiang 310058, China}
\email{guwei2875@zju.edu.cn}

\author{Leonardo C.~Mihalcea}
\address{
Department of Mathematics, 
225 Stanger Street, McBryde Hall,
Virginia Tech University, 
Blacksburg, VA 24061
USA
}
\email{lmihalce@vt.edu}

\author{Eric Sharpe}
\address{Department of Physics MC 0435, 850 West Campus Drive,
Virginia Tech University, Blacksburg VA  24061
USA}
\email{ersharpe@vt.edu}

\author{Weihong Xu}
\address{Division of Physics, Mathematics,  and Astronomy,
Caltech, 1200 E. California Blvd.,
Pasadena CA 91125  USA}
\email{weihong@caltech.edu}

\author{Hao Zhang}
\address{Department of Physics MC 0435, 850 West Campus Drive,
Virginia Tech University, Blacksburg VA  24061
USA}
\email{hzhang96@vt.edu}

\author{Hao Zou}
\address{Center for Mathematics and Interdisciplinary Sciences, Fudan University, Shanghai 200433, China}
 \address{Shanghai Institute for Mathematics and Interdisciplinary Sciences, Shanghai 200433, China}
\email{haozou@fudan.edu.cn}

\authormark{W.~Gu, L.\,C.~Mihalcea, E.~Sharpe, W.~Xu, H.~Zhang, H.~Zou}

\AbstractInEnglish{We prove that the ideal of relations in the
  (equivariant) quantum K-ring of a homogeneous space is generated by
  quantizations of each of the generators of the ideal in the
  classical (equivariant) K-ring. This extends to quantum K-theory a
  result of Siebert and Tian in quantum cohomology.  We illustrate
  this technique in the case of the quantum K-ring of partial flag
  manifolds, using a set of quantum K-Whitney relations conjectured by
  the authors, and recently proved by Huq-Kuruvilla.}


\acknowledgement{We thank the Simons Center for Geometry and Physics for supporting the excellent workshop GLSM@30, which made possible many stimulating discussions. W.\,G. was partially supported by NSF grant PHY-1720321. L.\,M. was partially supported by NSF grant DMS-2152294, and gratefully acknowledges the support of Charles Simonyi Endowment, which provided funding for the membership at the Institute of Advanced Study during the 2024--25 Special Year in `Algebraic and Geometric Combinatorics'.  E.\,S. was partially supported by NSF grant PHY-2310588. H.\,Zou  was partially supported by the National Natural Science Foundation of China (Grant No.~12405083, 12475005) and the Shanghai Magnolia Talent Program Pujiang Project (Grant No.~24PJA119).}

\begin{document}


\removeabove{0.8cm}
\removebetween{0.8cm}
\removebelow{0.8cm}

\maketitle

\begin{prelims}

\DisplayAbstractInEnglish

\bigskip

\DisplayKeyWords

\medskip

\DisplayMSCclass

\end{prelims}


\newpage

\setcounter{tocdepth}{1}

\tableofcontents


\section{Introduction}\label{sec:intro}
In the study of presentations by generators and relations of the quantum cohomology ring of a projective manifold, one of the most widely used results is that of Siebert and Tian \cite{siebert.tian:on}.  They proved that a generating set of the ideal of relations for the quantum cohomology ring may be obtained simply by quantizing each of the generators of the ideal of relations in the {\em ordinary} cohomology ring -- no other relations were needed.  The proof is based on a graded version of the Nakayama lemma; see also \cite{fulton.pandh:notes} and, in the equivariant case, \cite{mihalcea:giambelli}.  This technique was quite successful in obtaining presentations of quantum cohomology rings in many situations.

Our main goal in this paper is to extend Siebert and Tian's result to Givental and Lee's quantum K-theory ring of homogeneous spaces; see \cite{givental:onwdvv,lee:QK}.  The quantum K-theory ring is not graded, and a more careful analysis is needed, which, for instance, requires working over rings completed along the ideal of quantum parameters.  We then show that a similar statement as in quantum cohomology holds; \textit{cf.}~\Cref{thm:qktfg}.

Our approach to this problem expands on results from \cite{gu2022quantum}, where we studied the quantum K-ring of Grassmann manifolds, and where we proved a weaker version of the general Nakayama-type result from this paper, for certain rings of power series.  A key hypothesis from \cite{gu2022quantum} is that the presentation obtained by quantizing each of the classical generators satisfies a module-finiteness condition. For the quantum K-ring of Grassmannians, we found an {\em ad hoc} proof of this condition; a similar situation happened in \cite{maeno.naito.sagaki:QKideal}.  However, it turns out that the required finiteness holds very generally, including in the case of most interest to us -- the (equivariant) quantum K-ring of a homogeneous space $G/P$. See \Cref{sec:completions} for precise statements, especially \Cref{prop:mainNak}. It is not difficult to show that basic properties of the quantum K-rings of homogeneous spaces satisfy all hypotheses from this corollary, leading to our main application for quantum K-rings proved in \Cref{thm:qktfg}. Along the way we also prove a folklore result which states that the Schubert divisors generate the equivariant quantum K-ring of the complete flag manifold $G/B$ as an algebra; see \Cref{thm:main}.
  
In \Cref{sec:QKWhitney} we illustrate our methods in the case of the equivariant quantum K-ring of partial flag manifolds. We prove that the Whitney relations obtained from the tautological subbundles give the full presentation of the ordinary equivariant K-ring (similar to one by Lascoux \cite{lascoux:anneau}).  An explicit quantization of this set of relations was conjectured in \cite{gu2023quantum}. The conjecture was known for Grassmannians, see \cite{gu2022quantum}, and for incidence varieties $\Fl(1,n-1;n)$, see \cite{GMSXZZ:QKW}; recently, Huq-Kuruvilla proved the full conjecture in \cite{huq2024quantum}.  We use these `quantum K-Whitney' relations to deduce our presentation; see \Cref{conj2:lambda_y}.

A discussion on how this presentation relates to other ones in the literature is given at the end of \Cref{sec:QKpres}.  We further illustrate the Whitney presentation for $\QK_T(\Fl(3))$ in \Cref{sec:example}, and also include a discussion on the need for completions.

\subsection*{Acknowledgments}
We thank Anders Buch, Linda Chen, Elana Kalashnikov, Peter Koroteev, Y.\,P.~Lee, and Henry Liu for helpful discussions. Special thanks are due to Prof.~Satoshi Naito for many inspiring discussions, and for pointing out some references. We are also grateful to an anonymous referee for a careful reading and valuable suggestions.

\section{About finiteness of completed rings}\label{sec:completions}

For notions about completions, we refer to \cite[Chapter~10]{AM:intro}; see also \cite[Appendix A]{gu2022quantum}.  The main fact we utilize in this paper is the following result proved in \cite[Proposition~A.3]{gu2022quantum}.

\begin{prop}[Gu--Mihalcea--Sharpe--Zou]\label{prop:Nakiso}
Let $A$ be a Noetherian integral domain, and let $\mathfrak{a} \subset A$ be an ideal.~Assume that $A$ is complete in the $\mathfrak{a}$-adic topology.  Let $M,N$ be finitely generated $A$-modules.

Assume that the $A$-module $N$ and the $A/\mathfrak{a}$-module $N/\mathfrak{a}N$ are both free modules of the same rank $p< \infty$, and that we are given an $A$-module homomorphism $f\colon M \to N$ such that the induced $A/\mathfrak{a}$-module map $\overline{f}\colon M/\mathfrak{a} M \to N / \mathfrak{a}N$ is an isomorphism of $A/\mathfrak{a}$-modules.

Then $f$ is an isomorphism.  
\end{prop}

A key hypothesis needed in this proposition is the module finiteness of $M$.  We record in \Cref{prop:KNak} below a simple criterion for when this holds, in the case when $M$ is a quotient of a power series ring; this is the case for our main application.  One may also deduce \Cref{prop:KNak} directly from \cite[Exercise 7.4]{eisenbud:CAbook};\begin{footnote}{We thank Prof. S. Naito for providing us with this reference.}\end{footnote} for the convenience of the reader, we include a proof.

We start with the following general result proved in \cite[Theorem~8.4]{matsumura:commutative}; see also \cite[\href{https://stacks.math.columbia.edu/tag/031D}{Tag 031D}]{StacksProj}.

\begin{lemma}\label{lemma:fgstacks}
Let $A$ be a commutative ring, and let $\mathfrak{a}\subset A$ be an ideal.  Let $M$ be an $A$-module. Assume that $A$ is $\mathfrak{a}$-adically complete, $\bigcap_{n \ge 1} \mathfrak{a}^n M = (0)$, and that $M/\mathfrak{a}M$ is a finitely generated $A/\mathfrak{a}$-module.  Then $M$ is a finitely generated $A$-module.\end{lemma}

For a commutative ring $S$ with $1$, we denote by $\Jac(S)$ its Jacobson radical, \textit{i.e.}~the intersection of all its maximal ideals. It is proved in \cite[Proposition~1.9]{AM:intro} that $x \in \Jac(S)$ if and only if $1-xy$ is a unit in $S$ for all $y \in S$.

\begin{lemma}\label{lemma:fJac} 
Let $R,S$ be commutative rings with $1$ and $\pi\colon R \to S$ be a surjective ring homomorphism with $\pi(1)=1$. Then:
    \begin{enumerate}
        \item We have $\pi(\Jac(R)) \subseteq \Jac(S)$. 
        \item\label{item:ideal} If\, \(J\) is an ideal in \(R\), then $\pi(J)$ is an ideal in $S$.
    \end{enumerate}
\end{lemma}

\begin{proof} 
Let $x \in \Jac(R)$. Then \(1-xr\) is a unit in \(R\) for all \(r\in R\). This implies \(f(1-xr)=1-f(x)f(r)\) is a unit in \(S\). Since \(f\) is surjective, this means \(f(x)\in \Jac(S)\).

Part~\eqref{item:ideal} is immediate from the definitions.
\end{proof}

From now on, $S$ is a commutative Noetherian ring, and \(I\) is an ideal of the formal power series ring $S[\![q_1, \ldots, q_k]\!]$. Let
\[\pi\colon S[\![q_1, \ldots, q_k]\!] \lra M\coloneqq S[\![q_1, \ldots, q_k]\!]/I\]
be the projection. Let
\[ J\coloneqq \langle q_1, \ldots, q_k \rangle \subset S[\![q_1, \ldots, q_k]\!]. \]

\begin{lemma}\label{lemma:piJac}
  The ideal $\pi(J)$ is contained in the Jacobson radical of\, $M$.
\end{lemma}

\begin{proof}
By \cite[Proposition~10.15]{AM:intro} the ideal  $J$ is contained in the Jacobson radical of $S[\![q_1, \ldots, q_k]\!]$. Then the claim follows from \Cref{lemma:fJac}.\end{proof}

\begin{cor}\label{cor:separated-gen}
  We have that $\bigcap_{n \ge 1} \pi(J)^n = (0)$.\end{cor}

\begin{proof}
Note that $S[\![q_1, \ldots, q_k]\!]$ is Noetherian from \cite[Corollary~10.27]{AM:intro}. Then its quotient $M$ is also Noetherian, and by \Cref{lemma:piJac} we have that $\pi(J) \subset \Jac(M)$. The claim follows from a corollary of Krull's theorem, \cite[Corollary~10.19]{AM:intro}, applied to $M$ as a module over {$S[\![q_1, \ldots, q_k]\!]$} and the ideal $\pi(J)$.
\end{proof}

Let us further assume  that $S$ is an $R$-algebra for a Noetherian ring $R$. Let 
$$A\coloneqq R[\![q_1, \ldots, q_k]\!] \subset S[\![q_1, \ldots, q_k]\!]$$ with ideal $\mathfrak{a}=\langle q_1, \ldots, q_k\rangle \subset A$.

\begin{prop}\label{prop:KNak}
Recall that $M \coloneqq S[\![q_1, \ldots, q_k]\!]/I$.  If $M/\mathfrak{a}M$ is a finitely generated \(A/\mathfrak{a}\)-module, then
\begin{enumerate}
    \item\label{item:fg} $M$ is a finitely generated $A$-module;
    \item\label{item:complete} \(M\) is \(\mathfrak{a}\)-adically complete.
\end{enumerate}
\end{prop}

\begin{proof} 
Note that $A$ is $\mathfrak{a}$-adically complete, see \cite[Section~7.1]{eisenbud:CAbook}, and that
\[
    \bigcap_{n \ge 1} \mathfrak{a}^nM = \bigcap_{n \ge 1} J^n M= \bigcap_{n \ge 1} \pi(J)^n =(0)
\] 
by \Cref{cor:separated-gen}.  Then part~\eqref{item:fg} follows from \Cref{lemma:fgstacks}. Since $A$ is $\mathfrak{a}$-dically complete, it follows from \cite[Proposition~10.13]{AM:intro} that the $\mathfrak{a}$-adic completion of $M$ is $\widehat{M} = M \otimes_A \widehat{A}=M$, proving part~\eqref{item:complete}.
\end{proof}

We now combine Propositions~\ref{prop:Nakiso} and~\ref{prop:KNak} to obtain the following.

\begin{cor}\label{prop:mainNak}
  Let $R$ be a Noetherian integral domain, $S$ a \emph{Noetherian}  $R$-algebra, and $I \subset S[\![q_1, \ldots, q_k]\!]$ an ideal of the formal power series ring $S[\![q_1, \ldots, q_k]\!]$. Assume that we have a homomorphism of\, $R[\![q_1, \ldots, q_k]\!]$-algebras
\[ f\colon S[\![q_1, \ldots, q_k]\!]/I \lra N \]
such that: 
 \begin{enumerate}
    \item\label{p:mN-1} $N$ is a free $R[\![q_1, \ldots, q_k]\!]$-module of finite rank $p<\infty$;
    \item\label{p:mN-2} $f$ induces an isomorphism of\, $R$-algebras 
    \[ \overline{f}\colon S[\![q_1, \ldots, q_k]\!]/(\langle q_1, \ldots, q_k \rangle +I) \lra N/\langle q_1, \ldots, q_k \rangle.\]
\end{enumerate}
Then $f$ is an isomorphism.
\end{cor}

\begin{proof}
Hypotheses~\eqref{p:mN-1} and~\eqref{p:mN-2} imply that $N/\langle q_1, \ldots, q_k \rangle$ and $S[\![q_1, \ldots, q_k]\!]/(\langle q_1, \ldots, q_k \rangle +I)$ are both free $R$-modules of rank $p$. Then apply \Cref{prop:Nakiso} to the case $A:=R[\![q_1, \ldots, q_k]\!]$ and $M:=S[\![q_1, \ldots, q_k]\!]/I$, and the homomorphism $f\colon M \to N$. By \Cref{prop:KNak}, $M$ is a finitely generated $A$-module, thus the hypotheses of \Cref{prop:Nakiso} hold, giving that $f$ is an isomorphism.
\end{proof}

As the example below shows, working with polynomial rings in $q$ 
is not sufficient in this corollary.

\begin{example}
Consider $R=\C$, $S=\C[x]$, $I=\langle (1-q) x \rangle \subset S[q]$, $M=S[\![q]\!]/I$, and $N=\C[\![q]\!]$.  Since $1-q$ is invertible in $S[\![q]\!]$, it follows that there is a $R[\![q]\!]$-algebra isomorphism
\[ \overline{f}\colon M \lra N, \quad x \longmapsto 0. \]
Furthermore, modulo $q$, this is again an isomorphism of $R$-algebras. However,
\[ S[q]/I = \C[x,q]/\langle (1-q)x \rangle \]
has zero divisors, so it is not isomorphic to \(R[q]=\C[q]\). For another example arising from the quantum K-ring of $\Fl(3)$, see \Cref{sec:example} below.\end{example}

\section{Equivariant quantum K-theory of flag varieties}\label{sec:preliminaries} 

\subsection{Preliminaries}
In this section we recall some basic facts about the equivariant K-theory of a variety with a group action. For an introduction to equivariant K-theory, and more details, see \cite{chriss2009representation}.

Let $X$ be a smooth projective variety with an action of a linear algebraic group $G$. The equivariant K-theory ring $\K_G(X)$ is the Grothendieck ring generated by symbols $[E]$, where $E \to X$ is a $G$-equivariant vector bundle, modulo the relations $[E]=[E']+[E'']$ for any short exact sequence $0 \to E' \to E \to E'' \to 0$ of equivariant vector bundles. The additive ring structure is given by direct sum, and the multiplication is given by tensor products of vector bundles.  Since $X$ is smooth, any $G$-linearized coherent sheaf has a finite resolution by (equivariant) vector bundles, and the ring $\K_G(X)$ coincides with the Grothendieck group of $G$-linearized coherent sheaves on $X$. In particular, any $G$-linearized coherent sheaf $\mathcal{F}$ on $X$ determines a class $[\mathcal{F}] \in \K_G(X)$.  As an important special case, if $\Omega \subset X$ is a $G$-stable subscheme, then its structure sheaf determines a class $[\cO_\Omega] \in \K_G(X)$.  We shall abuse notation and sometimes write \(\cF\) or $\cO_\Omega$ for the corresponding classes $[\cF]$ and $[\cO_\Omega]$ in $\K_G(X)$.

The ring $\K_G(X)$ is an algebra over $\K_G(\pt) = \Rep(G)$, the representation ring of $G$. If $G=T$ is a complex torus, then this is the Laurent polynomial ring $\K_T(\pt) = \Z[T_1^{\pm 1}, \ldots, T_n^{\pm 1}]$, where the $T_i:=\C_{t_i}$ are characters corresponding to a basis of the Lie algebra of $T$.

Let $E \to X$ be an equivariant vector bundle of rank $\mathrm{rk}\,E$. The (Hirzebruch) $\lambda_y$-class is defined as
\[ \lambda_y(E) := 1 + y E + \ldots + y^{\mathrm{rk}\,E} \wedge^{\mathrm{rk}\,E} E \quad \in \K_T(X)[y]. \] 
This class was introduced by Hirzebruch \cite{hirzebruch:topological} in relation to the Grothendieck--Riemann--Roch theorem.  The $\lambda_y$ class is multiplicative with respect to short exact sequences; \textit{i.e.} if $0 \to E' \to E \to E'' \to 0$ is such a sequence of vector bundles, then
\[ \lambda_y(E) = \lambda_y(E') \cdot \lambda_y(E''). \] 
This is part of the $\lambda$-ring structure of $\K_T(X)$, see \textit{e.g.} \cite{nielsen:diag}, referring to \cite{SGA6}. 

A particular case of this construction is when $V$ is a (complex) vector space with an action of a complex torus $T$. The $\lambda_y$-class of $V$ is the element $\lambda_y(V) =\sum_{i \ge 0} y^i \wedge^i V \in \K_T(\pt)[y]$. Since \(V\) decomposes into $1$-dimensional $T$-representations,  $V = \bigoplus_i \C_{\mu_i}$, it follows from the multiplicative property of $\lambda_y$-classes that \(\lambda_y(V) = \prod_i (1+y\,\C_{\mu_i})\).

Since $X$ is proper, we can push the class of a sheaf forward to the point. This is given by the sheaf Euler characteristic or, equivalently, the virtual representation
\[\chi_X^T(\mathcal{F}) 
:= \sum_i (-1)^i \HH^i(X, \mathcal{F})\quad\in \K_T(\pt). \] 

\subsection{(Equivariant) K-theory of homogeneous spaces}\label{sec:kflag} 
Let $G$ be a complex, reductive algebraic group, and fix $T:=B \cap B^- \subset B \subset P \subset G$: a standard parabolic subgroup $P$ containing a Borel subgroup $B$ and the maximal torus $T$ obtained as the intersection of $B$ with the opposite Borel subgroup $B^-$. Let $W:=N_G(T)/T$ denote the Weyl group, equipped with the length function $\ell\colon W \to \mathbb{Z}$, and denote by $W_P$ the subgroup generated by the simple reflections $s_i$ with representatives in $P$. The set of minimal-length representatives for the cosets of $W/W_P$ is denoted by $W^P$; thus $W^B=W$.

The (generalized) flag variety is by definition $X=G/P$, a homogeneous space with respect to the left action of $G$. It has a stratification by (opposite) Schubert cells:
\[ G/P = \bigsqcup_{w \in W^P} X_w^\circ = \bigsqcup_{w \in W^P} X^{w,\circ}, \]
where $X_w^\circ := BwP/P \simeq \C^{\ell(w)}$ and $X^{w, \circ}:= B^- w P/P \simeq \C^{\dim G/P - \ell(w)}$.  The closures of the Schubert cells are the Schubert varieties: $X_w:= \overline{X_w^\circ},\ X^w :=\overline{X^{w,\circ}}$. These are $T$-stable and determine classes $\cO_w:=[\cO_{X_w}]$ and \(\cO^w=[\cO_{X^w}]\) in $\K_T(X)$. These classes form a $\K_T(\pt)$-basis for $\K_T(X)$; \textit{i.e.}
\[ \K_T(X) = \bigoplus_{w \in W^P} \K_T(\pt) \cO_w =  \bigoplus_{w \in W^P} \K_T(\pt) \cO^w. \]
In particular, $\K_T(X)$ is a finitely generated $\K_T(\pt)$-algebra: one generating set is given by the Schubert classes. Alternatively, the finite generation also follows from the Atiyah--Hirzebruch isomorphism of $\K_T(\pt)$-modules
\[ \K_T(G/P) \simeq \K_T(\pt) \otimes_{\K_G(\pt)} \K_T(\pt)^{W_P}; \] 
see \textit{e.g.}~\cite[Sections~6.1 and 6.2]{chriss2009representation} or \cite[Appendix]{MNS:left}.

For $X=G/B$ one can show that the Schubert divisors already generate. We include below a proof, for the convenience of the reader, and also because we could not find a precise reference of this result, undoubtedly known to experts.

\begin{prop}\label{thm:main}
Assume \(G\) is simply connected. Then the ring $\K_T(G/B)$ is generated over $\K_T(\pt)$ by the Schubert divisors $\cO^{s_i}$, where $\{ s_i \}$ is the set of simple reflections in $W$. \end{prop}

In order to prove this, we introduce more notation. Let $X^*(T)$ be the group of characters of $T$, written in the additive notation. For $\lambda \in X^*(T)$ we denote by $\C_\lambda$ the $1$-dimensional $T$-representation with character~$\lambda$. One may extend this to a $B$-representation by letting the unipotent group $U$ act trivially. Define the $G$-equivariant line bundle $\mathcal{L}_\lambda$ over $G/B$ by
\[ \mathcal{L}_\lambda = G \times^B \C_\lambda. \]
Denote by $\omega_i$ the fundamental weights, and set $\rho = \sum_i \omega_i$. There is an exact sequence
\begin{equation}\label{E:exact}
  0 \lra \mathcal{L}_{\omega_i} \otimes \C_{-\omega_i} \lra \cO_{G/B} \lra \cO^{s_i} \lra 0. \end{equation} 
A consequence of this exact sequence is that $\cO^{s_i}$ and the line bundles $\mathcal{L}_{\omega_i}$ generate the same subring of $\K_T(G/B)$.  Since the line bundle $\mathcal{L}_{-\rho}$ is very ample, there is a $T$-equivariant embedding
\[ \iota\colon G/B \longhookrightarrow \mathbb{P}(V) = \mathbb{P}\left(H^0\left(G/B; \mathcal{L}_{-\rho}\right)\right) \]
such that $\iota^*(\cO_{\bP(V)}(1)) = \mathcal{L}_{-\rho}$.

\begin{lemma}\label{lemma:O1}
The class of\, $\cO_{\bP(V)}(1) \in \K_T(\bP(V))$ may be written as a polynomial in the class $\cO_{\bP(V)}(-1)$ with coefficients in $\K_T(\pt)$.\end{lemma}

\begin{proof}
On $\bP(V)$ there is the tautological sequence of $T$-equivariant vector bundles
\[ 0 \lra \cO_{\bP(V)}(-1) \lra V \lra \cQ:=V/\cO_{\bP(V)}(-1) \lra 0. \]
The $\K$-theoretic Whitney relations are 
\[ \lambda_y(\cO_{\bP(V)}(-1)) \cdot \lambda_y(\cQ) = \lambda_y(V) = 1 + y V + y^2 \wedge^2 V + \cdots + y^N \wedge^N V, \]
where $N= \dim V$. Multiplying both sides by \(\sum_{i\geq 0}(-1)^iy^i(\cO_{\bP(V)}(-1))^i\) and extracting the coefficient of \(y^{N-1}\), one obtains that
\[ \det \cQ = \sum_{i=0}^{N-1} (-1)^i \wedge^{N-i-1} V  \otimes \left(\cO_{\bP(V)}(-1)\right)^i. \]
But $\cO_{\bP(V)}(-1) \otimes \det \cQ = \det V$; therefore, 
\begin{equation}\label{E:O1} \cO_{\bP(V)}(1) = \left(\cO_{\bP(V)}(-1)\right)^{-1} = \det \cQ \otimes (\det V)^{-1}.
\end{equation}
Finally, since $(\det V)^{-1} \in \K_T(\pt)$, the right-hand side of~\eqref{E:O1} has the claimed properties.
\end{proof}
 
\begin{proof}[Proof of \Cref{thm:main}]
Let $R$ be the subring of $\K_T(G/B)$ generated by the Schubert divisors. Observe that $\iota^*\cO_{\bP(V)}( \pm 1) = \mathcal{L}_{\mp \rho}$. From \Cref{lemma:O1} it follows that $\mathcal{L}_{-\rho}$ is a polynomial in powers of $\mathcal{L}_\rho$ with coefficients in $\K_T(\pt)$; thus $\mathcal{L}_{-\rho} \in R$. Then for each fundamental weight $\omega_i$,
\[ \mathcal{L}_{-\omega_i} = \mathcal{L}_{-\rho} \otimes \mathcal{L}_{\omega_1} \otimes \cdots \otimes \widehat{\mathcal{L}_{\omega_i}} 
\otimes \cdots \otimes \mathcal{L}_{\omega_{\mathrm{rank}~G}} \] is an element of $R$. Since every line bundle is generated by powers of $\mathcal{L}_{\pm \omega_i}$, this shows that $R$ is equal to the subring generated over $\K_T(\pt)$ by all  $G$-equivariant line bundles. Our initial hypothesis that $G$ is simply connected implies that these line bundles generate the whole $\K_T(G/B)$; see \textit{e.g.}~\cite[Lemma~6.1.6 and Theorem~6.1.22]{chriss2009representation}.\begin{footnote}{Note that this is the only place where the simply connected hypothesis is needed.}\end{footnote} This finishes the proof.
\end{proof}

\begin{remark}
In general the Schubert divisors do not generate the $\K$-theory ring of $G/P$.  For example, in $\K(\Gr(2,4))$, the minimal polynomial of the multiplication by the Schubert divisor has degree $5$, while the K-theory module has rank $6$. However, it can be shown that the Schubert divisor classes generate an appropriate localization of $\K_T(G/P)$; \textit{cf.} \cite[Section~5.3]{BCMP:qkchev}.\end{remark}

\subsection{Quantum K-theory}
We next recall  the definition of the equivariant quantum K-ring of a partial flag variety.  An effective degree is a \(k\)-tuple of nonnegative integers \(\mathbf{d}=(d_1,\dots,d_k)\), which is identified with \(\sum_{i=1}^{k}d_i[X_{s_{r_i}}]\in\HH_2(X,\Z)\).  We write \(\mathbf{q}^\mathbf{d}\) for \(q_1^{d_1}\dots q_k^{d_k}\), where $\mathbf{q}=(q_1, \ldots, q_k)$ is a sequence of quantum parameters.

We recall the definition of the \(T\)-equivariant (small) quantum K-theory ring $\QK_T(X)$, following \cite{givental:onwdvv,lee:QK}. Additively,
\begin{equation*}\label{eqn:QKTfree}
\QK_T(X) = \K_T(X) \otimes_{\K_T(\pt)} \K_T(\pt)[\![\mathbf q]\!] 
\end{equation*}
is a free $\K_T(\pt)[\![\mathbf q]\!]$-module with a $\K_T(\pt)[\![\mathbf q]\!]$-basis given by Schubert classes $\{ \cO^w \}_{w \in W^P}$.  It is equipped with a commutative, associative product, denoted by $\star$, and determined by the condition
\begin{equation}\label{eqn:def}
    (\!(\sigma_1\star\sigma_2,\sigma_3)\!)=\sum_d
        \mathbf{q}^\mathbf{d} \egw{\sigma_1,\sigma_2,\sigma_3}{d}
\end{equation}
for all \(\sigma_1,\sigma_2,\sigma_3\in \K_T(X)\), where  
\[
        (\!(\sigma_1,\sigma_2)\!)\coloneqq \sum_{d}\mathbf{q}^\mathbf{d} \egw{\sigma_1,\sigma_2}{d}
\]
is the quantum \(\K\)-metric and \(\egw{\sigma_1,\dots,\sigma_n}{d}\) are \(T\)-equivariant \(\K\)-theoretic Gromov--Witten invariants, recalled next. Let \(d\in \HH_2(X,\Z)_+\) be an effective degree, and let \(\Mb_{0,n}(X,d)\) be the Kontsevich moduli space parametrizing \(n\)-pointed, genus \(0\), degree \(d\) stable maps to \(X\). Denote by
\[
\ev_1, \ev_2, \dots, \ev_n\colon \Mb_{0,n}(X,d)\lra X
\] 
the evaluations at the \(n\) marked points. 
Given \(\sigma_1, \sigma_2,\dots, \sigma_n\in \K_T(X)\), the \(T\)-equivariant \(\K\)-theoretic Gromov--Witten invariant is defined by \[
\egw{\sigma_1,\sigma_2,\dots,\sigma_n}{d}\coloneqq\chi_{\M_{0,n}(X,d)}^T(\ev_1^*\sigma_1\cdot\ev_2^*\sigma_2\cdots\ev_n^*\sigma_n),
\] 
where \(\chi_Y^T\colon \K_T(Y)\to \K_T(\pt)\) is the pushforward to a point. We adopt the convention that when \(d\) is not effective, the invariant \(\egw{\sigma_1,\sigma_2,\dots,\sigma_n}{d}\) is zero.

\begin{remark}
For $1 \le j \le k$ set $\deg q_{j} = \deg ([X_{s_{r_j}}] \cap c_1(T_X)) = r_{j+1}- r_{j-1}$.  For a multidegree $d = (d_{r_1}, \ldots, d_{r_k})$, set $\deg(\cO^w \otimes \mathbf{q}^\mathbf{d}) = \ell(w) + \sum \deg(q_i) \cdot d_{r_i}$.  Together with the topological filtration on $\K_T(X)$, this equips $\QK_T(X)$ with the structure of a filtered ring; see \cite[Section~5.1]{buch.m:qk}. The associated graded of this ring is $\QH^*_T(X)$, the (small) $T$-equivariant quantum cohomology of $X$, a free $\HH^*_T(\pt)[\mathbf q]$-algebra of the same rank as $\QK_T(X)$. \end{remark} 

\section{Relations of the quantum K-theory ring}
As above, we let $X:=G/P$. Our goal in this section is to prove \Cref{thm:qktfg}, giving the K-theoretic version of Siebert and Tian's results in quantum cohomology.

Denote by $\mathbf{q}=(q_1, \ldots, q_k)$ the sequence of quantum parameters associated to $X$.

\begin{thm}\label{thm:qktfg}
Assume that we are given an isomorphism of\, $\K_T(\pt)$-algebras
\[ \Phi\colon \K_T(\pt)[m_1, \ldots, m_s]/\langle P_1, \ldots , P_r \rangle \lra \K_T(X), \]
where $m_1, \ldots, m_s$ are indeterminates, and the $P_i$ are polynomials in the variables $m_j$ with coefficients in $\K_T(\pt)$.  Consider any power series
\[ \tilde{P}_i =\tilde{P}_i(m_1, \ldots, m_s;q_1, \ldots, q_k) \in \K_T(pt)[m_1, \ldots, m_s][\![q_1, \ldots, q_k]\!] \]
such that:
\begin{itemize}
\item  $\tilde{P}_i(m_1, \ldots, m_s; 0, \ldots, 0) = P_i(m_1,\ldots, m_s)$ for all $m_1, \ldots, m_s$ and all $i$, and 
\item $\tilde{P}_i(m_1, \ldots, m_s;q_1, \ldots, q_k)=0$ in $\QK_T(X)$, 
where the $m_i$ are regarded as elements in $\QK_T(X)$ via the isomorphism $\Phi$.
\end{itemize}
Then $\Phi$ lifts to an isomorphism of\, $\K_T(\pt)[\![\mathbf{q}]\!]$-algebras
\[ \widetilde{\Phi}\colon  \K_T(\pt)[\![\mathbf{q}]\!][m_1, \ldots, m_s]/\langle \tilde{P}_1, \ldots , \tilde{P}_r \rangle \lra \QK_T(X).\]
\end{thm}

\begin{proof}
We use \Cref{prop:mainNak} applied to $N= \QK_T(X)$, the Noetherian integral domains $R=\K_T(\pt)$ and $S=R[m_1, \ldots, m_s]$, and $I$ the ideal $\langle \tilde{P}_1, \ldots , \tilde{P}_r \rangle \subset S[\![q_1, \ldots, q_k]\!]$.  By the hypothesis on the $\tilde{P}_i$, $\widetilde{\Phi}$ is a well-defined $R[\![q_1, \ldots, q_k]\!]$-algebra homomorphism, and the induced homomorphism modulo $\langle q_1, \ldots, q_k\rangle$ is an isomorphism.  Finally, $N$ is a finitely generated free module over $R[\![q_1, \ldots, q_k]\!]$, from the definition of the quantum K-ring.  Then all hypotheses of \Cref{prop:mainNak} are satisfied; thus $\widetilde{\Phi}$ is an isomorphism, as claimed. \end{proof}

As an aside, we also record the following lemma. Together with \Cref{thm:main}, it implies that the Schubert divisors generate $\QK_T(G/B)$ as an algebra over $\K_T(\pt)[\![\mathbf{q}]\!]$. A Toda-type presentation for $\QK_T(\Fl(n))$ with closely related generators has been obtained by Maeno, Naito, and Sagaki \cite{maeno.naito.sagaki:QKideal} -- see also \cite{lenart.maeno:quantum} and the next section.

\begin{lemma}\label{lemma:QKT-fgalgebra}
Assume that $m_1, \ldots, m_s$ generate $\K_T(X)$ as an algebra over $\K_T(pt)$. Then $m_1, \ldots, m_s$ also generate $\QK_T(X)=\K_T(X) \otimes_{\K_T(\pt)} \K_T(pt)[\![\mathbf{q}]\!]$ as an algebra over $\K_T(pt)[\![\mathbf{q}]\!]$.\end{lemma}

\begin{proof}
The argument follows from \cite[Lemma A.1]{gu2022quantum}.  Set $R:=\K_T(\pt)$, $M:=\QK_T(X)$, and let $I$ be the ideal in $R[\![\mathbf{q}]\!]$ generated by $q_1, \ldots, q_k$. Denote by $M'$ the $R[\![\mathbf{q}]\!]$-submodule of $M$ generated by $m_1, \ldots, m_k$.  Note that $R[\![\mathbf{q}]\!]$ is Noetherian, and $I$-adically complete, so $I$ is included in the Jacobson radical of $R[\![\mathbf{q}]\!]$; see \cite[Proposition~10.15]{AM:intro}.  The hypothesis in the lemma implies that $M/M'= I\cdot(M/M')$ as $R[\![\mathbf{q}]\!]$-modules.  Since $M$ is a finitely generated over $R[\![\mathbf{q}]\!]$, the claim follows from the usual Nakayama lemma, as stated \textit{e.g.}~in \cite[Proposition~2.6]{AM:intro}.
\end{proof}

\begin{remark}
Given the polynomials $P_i=P_i(m_1, \ldots, m_s)$ from \Cref{thm:qktfg}, one way to construct the power series $\tilde{P}_i$ is as follows.  First, observe that there is a surjective $\K_T(\pt)[\![q]\!]$-algebra homomorphism
\[ \Psi\colon \K_T(\pt)[\![\mathbf{q}]\!][m_1, \ldots, m_s] \lra \QK_T(X).\]
Let $G_u \in \K_T(\pt)[\![\mathbf{q}]\!][m_1, \ldots, m_s]$ be any preimage of the Schubert class $\cO_u$ under $\Psi$, for $u$ varying in the appropriate subset of the Weyl group. Regard $P_i$ as an element in $\K_T(\pt)[\![\mathbf{q}]\!][m_1, \ldots, m_s]$.  The hypothesis that $P_i$ gives a relation in $\K_T(X)$ implies that
\[ \Psi(P_i) = \sum_u a_u \cO_u, \]
where $a_u \in \K_T(\pt)[\![\mathbf{q}]\!]$ and $q_i | a_u$ for some index $i$. Define
\[ \tilde{P}_i:= P_i - \sum a_u G_u \quad \in \K_T(\pt)[\![\mathbf{q}]\!][m_1, \ldots, m_s]. \]
Then $\tilde{P}_i$ is in $\Ker(\Psi)$ and  satisfies the two conditions in \Cref{thm:qktfg}.  Note that the same proof holds in the more general case when $\K_T(X)$ has a finite $\K_T(\pt)$-basis.
\end{remark}

\section{The Whitney presentation of the quantum K-theory of partial flag manifolds}\label{sec:QKWhitney}
In this section we illustrate our method and prove a presentation of the equivariant quantum K-ring of the partial flag manifolds $X=\Fl(r_1, \ldots, r_k; \C^n)$.  To this aim, we first prove a (folklore) presentation of $\K_T(X)$, and then we recall a conjecture from \cite{gu2023quantum} giving a quantization of the generators of the ideal of classical relations; see \Cref{conj2:lambda_y}.~The conjectured quantizations were recently proved by Huq-Kuruvilla in \cite{huq2024quantum}, and by \Cref{thm:qktfg} they generate the full ideal of quantum K-relations.
 
For Grassmannians, the presentation resulting from the relations in \Cref{conj2:lambda_y} was proved in \cite{gu2022quantum}, and for the incidence varieties $\Fl(1,n-1;n)$, it was proved in an earlier ar$\chi$iv version of this paper \cite{GMSXZZ:QKW}.  We note that for $\QK_T(\Fl(n))$ we provided in \cite{GMSXZZ:QKW} a different proof of the quantum K\nobreakdash-relations, contingent on the `K-theoretic divisor axiom' -- a conjectural statement by Buch and Mihalcea which calculates the $3$-point K-theoretic GW invariants of the form $\egw{\cO^{s_i},\cO_u,\cO^v}{d}$.

For similar results in the case of the symplectic flag manifold $\Sp(2n)/B$, see \cite{kouno.naito:C}.

\subsection{The (classical) Whitney presentation}
Let $0=\cS_{0}\subset\cS_{1}\subset\dots \subset \cS_{k}\subset\cS_{{k+1}}=\C^n$ be the flag of tautological vector bundles on $X$, where \(\cS_{j}\) has rank \(r_j\).  Since we could not find a precise reference, we will take the opportunity to outline a proof for a (folklore) presentation by generators and relations of $\K_T(X)$. The relations
\begin{equation*}
    \lambda_y(\cS_{j})\cdot\lambda_y(\cS_{{j+1}}/\cS_{j})=\lambda_y(\cS_{{j+1}}),\quad j=1,\dots,k
\end{equation*}
arise from the Whitney relations applied to the exact sequences
\[
    0 \lra \cS_{{j}} \lra \cS_{j+1} \lra \cS_{j+1}/\cS_{{j}} \lra 0, \quad j=1,\dots,k .
\] 
We use the Whitney relations to obtain a presentation closely related to well-known (nonequivariant) presentations such as those in \cite[Section~7]{lascoux:anneau}.  Let
\[
    {X}^{(j)}=\left(X^{(j)}_1,\dots,X^{(j)}_{r_j}\right)\quad\text{and}\quad{Y}^{(j)}=\left(Y^{(j)}_1,\dots, {Y^{(j)}_{r_{j+1}-r_j}}\right)
\] 
denote formal variables for \(j=1,\dots,k\).  Let \(X^{(k+1)}\coloneqq (T_1,\dots,T_n)\) be the equivariant parameters in \(\K^T(\pt)\). Geometrically, the variables $X_i^{(j)}$ and $Y_s^{(j)}$ arise from the splitting principle:
\[ \lambda_y(\cS_j) = \prod_i \left(1+y X_i^{(j)}\right), \quad  \lambda_y(\cS_{j+1}/\cS_{j}) = \prod_s \left(1+y Y_s^{(j)}\right); \]
\textit{i.e.} they are the K-theoretic Chern roots of $\cS_{j}$ and $\cS_{j+1}/\cS_{{j}}$, respectively. Let \(e_\ell({X}^{(j)})\) and \(e_\ell({Y}^{(j)})\) be the \(\supth{\ell}\) elementary symmetric polynomials in \({X}^{(j)}\) and \({Y}^{(j)}\), respectively.  Denote by $S$ the (Laurent) polynomial ring
\[ 
 S\coloneqq\K_T(\pt)\left[e_i\left(X^{(j)}\right), e_s\left(Y^{(j)}\right); 1 \le j \le k, 1 \le i \le r_j, 1 \le s \le {r_{j+1}-r_j}\right], 
\]
and define the ideal $I \subset S$ generated by
\begin{equation}\label{eqn:rel_classical}
    \sum_{i+s=\ell} e_i\left(X^{(j)}\right) e_s\left(Y^{(j)}\right) - e_\ell\left(X^{(j+1)}\right), \quad 1 \le j \le k \text { and }1\leq \ell\leq r_{j+1}. 
\end{equation}

\begin{prop}\label{prop:classical-Whitney}
There is an isomorphism of\, $\K_T(\pt)$-algebras 
\[ \Psi\colon S/I \lra \K_T(X) \]
sending $e_i(X^{(j)})$ to $\wedge^i \cS_{j}$ and $e_i(Y^{(j)})$ to $\wedge^i {(\cS_{{j+1}}/\cS_{{j}})}$.
\end{prop}

\begin{proof}
Denote the conjectured presentation ring by $A$. The K-theoretic Whitney relations imply that $\lambda_y(\cS_j) \cdot \lambda_y(\cS_{j+1}/\cS_j) = \lambda_y(\cS_{j+1})$.  Then the geometric interpretation of the variables $X^{(j)}, Y^{(j)}$ in terms of the splitting principle before the theorem implies that $\Psi$ is a well-defined $\K_T(\pt)$-algebra homomorphism.

To prove the surjectivity of $\Psi$, we first consider the case when $X=\Fl(n)$ is the full flag variety, and we utilize the theory of double Grothendieck polynomials; see \cite{fulton.lascoux,buch:grothendieck}. It was proved in \cite[Theorem~2.1]{buch:grothendieck} that each Schubert class in $\K_T(X)$ may be written as a (double Grothendieck) polynomial in
\begin{equation*}\label{E:vars} 
1-(\C^n/\cS_{n-1})^{-1},\ 1- (\cS_{n-1}/\cS_{n-2})^{-1},\ \ldots\ ,\ 1 - (\cS_2/\cS_1)^{-1},\ 1-(\cS_1)^{-1} 
\end{equation*}
with coefficients in $\K_T(\pt)$. Note that in $\K_T(X)$  
\begin{align*}
    (\cS_{i}/\cS_{i-1})^{-1}
    =&\det(\C^n)^{-1}\cdot\C^n/\cS_{n-1}\cdots{\cS_{i+1}/\cS_{i}}\cdot\cS_{i-1}/\cS_{i-2}\cdots\cS_2/\cS_1\cdot\cS_1 
\end{align*}
for \(i=1,\dots,n\).  Therefore, each Schubert class may be written as a polynomial in variables \(\cS_{i}/\cS_{i-1}\) for \(i=1,\dots,n\) with coefficients in \(\K_T(\pt)\). This proves the surjectivity in this case.

For partial flag varieties, consider the injective ring homomorphism given by pulling back via the natural projection $p\colon \Fl(n) \to \Fl(r_1, \ldots, r_k;n)$. The pullbacks of Schubert classes and of the tautological bundles are
\[ p^*\cO^w = \cO^w, \quad p^*\left(\wedge^\ell \cS_i\right) = \wedge^\ell \cS_{r_i}, \quad 
p^*\left(\wedge^\ell (\cS_i/\cS_{i-1})\right) = \wedge^\ell \left(\cS_{r_i}/\cS_{r_{i-1}}\right) \] for any $w \in S_n^{r_1, \ldots, r_k}$, any $ 1 \le i \le k$, and any $\ell$; here $S_n^{r_1, \ldots, r_k}$ denotes the set of minimal-length representatives of elements in the quotient $S_n/(S_{r_1} \times S_{r_2-r_1} \ldots \times S_{n-r_k})$.  On the other hand, since $w \in S_n^{r_1, \ldots, r_k}$, the Schubert classes $p^*\cO^w$ may be written as (double Grothendieck) polynomials symmetric in each block of variables $1-(\cS_{r_i+1}/\cS_{r_i})^{-1}, \ldots, 1-(\cS_{r_{i+1}}/\cS_{r_{i+1}-1})^{-1}$, for $0 \le i \le k$, that is, in the elementary symmetric functions $e_{\ell}((1-(\cS_{r_i+1}/\cS_{r_i})^{-1}, \ldots, 1-(\cS_{r_{i+1}}/\cS_{r_{i+1}-1})^{-1}))$ in these sets of variables. Each such elementary symmetric function may be further expanded as a $\Z$-linear combination of terms of the form
\begin{equation*}\label{E:fractions}
  \frac{e_{s}\left(\cS_{r_i+1}/\cS_{r_i}, \ldots, \cS_{r_{i+1}}/\cS_{r_{i+1}-1}\right)}{\cS_{r_i+1}/\cS_{r_i}\cdot \ldots \cdot \cS_{r_{i+1}}/\cS_{r_{i+1}-1}} = \frac{\wedge^s\left(\cS_{r_{i+1}}/\cS_{r_i}\right)}{\det \left(\cS_{r_{i+1}}\right) / \det \left(\cS_{r_i}\right)}. \end{equation*}
Observe that $\det \C^n = \prod_{i=1}^{k+1} {\det} (\cS_{r_i}/\cS_{r_{i-1}}) {= \det(\cS_{r_{j}}) \prod_{i=j+1}^{k+1}\det(\cS_{r_i}/\cS_{r_{i-1}})}$.  Therefore,
\[ \det\left( \cS_{r_{j}}\right)^{-1} = (\det \C^n)^{-1} {\prod_{i=j+1}^{k+1}\det\left(\cS_{r_i}/\cS_{r_{i-1}}\right)}, \quad j=1,\dots, k. \]
We have shown that the pullbacks of Schubert classes $p^*(\cO^w)$ are polynomials in the pullbacks of the tautological bundles and their quotients, and 
we deduce that $\Psi$ is surjective for partial flag manifolds.

Injectivity holds because $\K_T(\pt)$ is an integral domain and both $A$ and $\K_T(X)$ have the same rank as free modules over $\K_T(\pt)$.  To see the latter, consider the ring
\[ A':= \Z[T_1, \ldots, T_n]\left[e_i\left(X^j\right), e_s(Y^j): 1 \le j \le k, 1 \le i \le r_j, 1 \le s \le 
{r_{j+1}-r_j}\right]/I' , \] where \(I'\subseteq A'\) is the ideal generated by~\eqref{eqn:rel_classical}.  It is classically known that the $\Z[T_1, \ldots, T_n]$-algebra $A'$ is isomorphic to the equivariant cohomology algebra $\HH^*_T(X)$, with $e_i(X^{(j)})$ being sent to the equivariant Chern class $c_i^T(\cS_j)$ and $e_i(Y^{(j)})$ to the equivariant Chern class $c_i^T(\cS_{j+1}/\cS_j)$. (This follows for example by realizing the partial flag variety as a tower of Grassmann bundles, then using a description of the cohomology of the latter as in \cite[Example 14.6.6]{fulton:IT}.) In particular, $A'$ is a free $\Z[T_1, \ldots, T_n]$-algebra of rank equal to the number of Schubert classes in $X$. Then $A= A' \otimes_{\Z[T_1, \ldots, T_n]} \K_T(\pt)$ is a free $\K_T(\pt)$-module of the same rank.
\end{proof}

\subsection{The quantum K-Whitney presentation}\label{sec:QKpres}
We keep the notation from the previous subsection.  The following was conjectured in \cite{gu2023quantum}. 

\begin{thm}\label{conj2:lambda_y}
For \(j=1,\dots, k\) the following relations hold in $\QK_T(X)$:
        \begin{equation}\label{eqn:lambda_y_rel}
            \lambda_y(\cS_{j})\star\lambda_y(\cS_{{j+1}}/\cS_{j})=\lambda_y(\cS_{{j+1}})-y^{r_{j+1}-r_j}\frac{q_j}{1-q_j}\det(\cS_{{j+1}}/\cS_{j})\star(\lambda_y(\cS_{j})-\lambda_y(\cS_{{j-1}})).
        \end{equation} 
\end{thm} 

These relations were proved for Grassmannians in \cite{gu2022quantum}, and for the incidence varieties $\Fl(1,n-1;n)$ in \cite{GMSXZZ:QKW}. The full conjecture was recently proved in \cite{huq2024quantum}.\begin{footnote}{Added in revision: a different proof was obtained recently by combining the K-theoretic divisor axiom in \cite{LNSX:QKdiv} with the earlier papers \cite{GMSXZZ:QKW,AHMOX:QKW}.}\end{footnote}

To get an abstract presentation, we start by transforming~\eqref{eqn:lambda_y_rel}. As in Section~\ref{sec:kflag}, let
\[
    {X}^{(j)}=\left(X^{(j)}_1,\dots,X^{(j)}_{r_j}\right)\quad\text{and}\quad{Y}^{(j)}=\left(Y^{(j)}_1,\dots, Y^{(j)}_{r_{j+1}-r_j}\right)
\] 
denote formal variables for \(j=1,\dots,k\). Let \(X^{(k+1)}\coloneqq (T_1,\dots,T_n)\) be the equivariant parameters in \(\K_T(\pt)\). Let \(e_\ell({X}^{(j)})\) and \(e_\ell({Y}^{(j)})\) be the \(\supth{\ell}\) elementary symmetric polynomials in \({X}^{(j)}\) and \({Y}^{(j)}\), respectively.

\begin{defn}\label{defn:Iq}
Let
    \[
        S=\K_T(\pt)\left[e_1\left(X^{(j)}\right),\dots, e_{r_j}\left(X^{(j)}\right),e_1\left(Y^{(j)}\right),\dots,e_{r_{j+1}-r_j}\left(Y^{(j)}\right), j=1,\dots, k\right].
    \]
    Let \(I_q\subset S[\![\mathbf q]\!]=S[\![q_1,\dots,q_k]\!]\) be the ideal generated by the coefficients of \(y\) in
\begin{multline*}\label{eqn:qkrel}
  \prod_{\ell=1}^{r_j}\left(1+y X^{(j)}_\ell\right)\prod_{\ell=1}^{r_{j+1}-r_j}\left(1+y Y^{(j)}_\ell\right)-\prod_{\ell=1}^{r_{j+1}}\left(1+y X^{(j+1)}_\ell\right)\\
  +y^{{r_{j+1}-r_j}}\frac{q_j}{1-q_j}\prod_{\ell=1}^{r_{j+1}-r_j}Y^{(j)}_\ell\left(\prod_{\ell=1}^{r_j}\left(1+yX^{(j)}_\ell\right)-\prod_{\ell=1}^{r_{j-1}}\left(1+yX^{(j-1)}_\ell\right)\right),\quad j=1,\dots, k.
\end{multline*}
\end{defn}

\begin{cor}\label{thm:all_relations}
There is an isomorphism of\, \(\K_T(\pt)[\![\mathbf q]\!]\)-algebras
    \[
        \Psi\colon {S[\![\mathbf q]\!]}/I_q\lra \QK_T(X)
    \]
    sending \(e_\ell(X^{(j)})\) to \(\wedge^\ell(\cS_{j})\) and \(e_\ell(Y^{(j)})\) to \(\wedge^\ell(\cS_{{j+1}}/\cS_{j})\).
    \end{cor}

\begin{proof}
When  $q_i=0$ for all $i$, the generators of the ideal $I_q$ are those of the ideal $I$ in \Cref{prop:classical-Whitney}. Then the claim follows from Theorems~\ref{conj2:lambda_y} and~\ref{thm:qktfg}.\end{proof}

As explained in \cite{gu2023quantum}, the Whitney presentation above may be viewed as a \(\K\)-theoretic analogue of a presentation obtained by Gu and Kalashnikov \cite{GuKa} of the quantum cohomology ring of quiver flag varieties.  A presentation of the equivariant quantum K-ring of the complete flag varieties $\Fl(n)$ was recently proved by Maeno, Naito, and Sagaki \cite{maeno.naito.sagaki:QKideal} (see also \cite{lenart.maeno:quantum}). It is related to the Toda lattice presentation from quantum cohomology; see \cite{givental.kim, kim.toda}.  Their presentation is built onto the relations satisfied by the $\lambda_y$-classes $\lambda_y(\cS_j/\cS_{j-1})$ of the quotient bundles.  By eliminating the $\lambda_y$-classes $\lambda_y(\cS_j)$, it is not difficult to show that the Whitney relations above for $\Fl(n)$ imply the presentation from \cite{maeno.naito.sagaki:QKideal}.

Relations similar to those from \cite{maeno.naito.sagaki:QKideal,givental.kim} also appear in \cite{givental.lee:quantum,anderson.et.al:finiteI}, in relation to the finite-difference Toda lattice; in \cite{koroteev}, in relation to the study of the quasimap quantum K-theory of the cotangent bundle of $\Fl(n)$ and the Bethe ansatz (see also \cite{Gorbounov:2014}); and in \cite{ikeda.iwao.maeno}, in relation to the Peterson isomorphism in quantum K-theory.

\mathversion{bold}
\subsection{An example for $\QK_T(\Fl(3))$}\label{sec:example}
\mathversion{normal}

We illustrate the presentation above for $\QK_T(\Fl(3))$.  There are two quantum K-Whitney relations:
\[ \begin{split} \lambda_y(\cS_1) \star \lambda_y(\cS_2/\cS_1) & = \lambda_y(\cS_2) - y \frac{q_1}{1-q_1} \det (\cS_2/\cS_1) \star (\lambda_y(\cS_1)-1), \\
  \lambda_y(\cS_2) \star \lambda_y(\C^3/\cS_2) & = \lambda_y(\C^3) - y \frac{q_2}{1-q_2} \det (\C^3/\cS_2) \star (\lambda_y(\cS_2)-\lambda_y(\cS_1)). \end{split} \]
In terms of abstract variables, these relations read
\[
\begin{split}
  \left(1+y X^{(1)}_1\right) \left(1+y Y^{(1)}_1\right)  & = \left(1+y X^{(2)}_1\right) \left(1+y X^{(2)}_2\right) - y^2 \frac{q_1}{1-q_1} X^{(1)}_1 Y^{(1)}_1,  \\
  \left(1+y X^{(2)}_1\right) \left(1+y X^{(2)}_2\right) \left(1+y Y^{(2)}_1\right)  & = (1+ yT_1)(1+yT_2)(1+yT_3) \\
  & \hphantom{=}\; -y \frac{q_2}{1-q_2} Y^{(2)}_1
 \left(\left(1+y X^{(2)}_1\right) \left(1+y X^{(2)}_2\right)-\left(1+y X^{(1)}_1\right)\right).
\end{split} \]
Identifying powers of $y$, we obtain the following:
\begin{gather*}
  X_1^{(1)}+Y_1^{(1)}-e_1(X^{(2)}),\\
  \frac{1}{1-q_1}X_1^{(1)}Y_1^{(1)}-e_2(X^{(2)}),\\
  e_1(X^{(2)})+\frac{1}{1-q_2}Y_1^{(2)}-e_1(T),\\
  \frac{1}{1-q_2}\left(e_1\left(X^{(2)}\right)-q_2X_1^{(1)}\right)Y_1^{(2)}-\left(e_2(T)-e_2\left(X^{(2)}\right)\right),\\
   \frac{1}{1-q_2}e_2\left(X^{(2)}\right)Y_1^{(2)}-e_3(T).
\end{gather*}
These generate the ideal $I_q$ of relations in $\QK_T(\Fl(3))$.

Note that $1-q_i$ is invertible in the ground ring $\K_T(\pt)[\![q_1,q_2]\!]=R[\![q_1,q_2]\!]$, so one may multiply by factors $(1-q_i)$ to get the following polynomial version of the relations:
 \begin{gather*}
            X_1^{(1)}+Y_1^{(1)}-e_1(X^{(2)}),\\
            X_1^{(1)}Y_1^{(1)}-(1-q_1)e_2(X^{(2)}),\\
            (1-q_2)\left(e_1\left(X^{(2)}\right)+Y_1^{(2)}-e_1(T)\right),\\
            \left(e_1\left(X^{(2)}\right)-q_2X_1^{(1)}\right)Y_1^{(2)}-(1-q_2)\left(e_2(T)-e_2\left(X^{(2)}\right)\right),\\
            e_2\left(X^{(2)}\right)Y_1^{(2)}-(1-q_2)e_3(T).
        \end{gather*}
Denote by $I_q^\poly$ the ideal of $S[q_1, q_2]$ generated by these relations. There is a natural ring homomorphism
 \[ \Phi^\poly\colon S[q_1,q_2]/I_q^\poly\lra S[\![q_1,q_2]\!]/I_q  \simeq \QK_T(\Fl(3)).\]
Using the third relation in $I_q^\poly$, one checks that \(e_1(X^{(2)})+Y_1^{(2)}-e_1(T)\) is a nonzero element in the kernel of~\(\Phi^\poly\). This shows that working with completions (or at least using a localized ring where $1-q_i$ is invertible) is necessary in \Cref{thm:qktfg}.
 
Furthermore, one may prove that $\Phi^\poly$ is surjective, as follows.  Let \(\QK_T^\poly(X)\subseteq\QK_T(X)\) be the subring generated by \(\cO^{s_1}\) and \(\cO^{s_2}\) over the ground ring \(\K_T(\pt)[q_1,q_2]\). Algorithm 4.16 of \cite{xu2021quantum} gives an algorithm that recursively expresses any Schubert class as a polynomial in \(\cO^{s_1},\ \cO^{s_2}\) with coefficients in \(\K_T(\pt)[q_1,q_2]\). Combined with \cite[Theorem~4.5]{xu2021quantum}, this means that when we express the product of two Schubert classes as a linear combination of Schubert classes in \(\QK_T(X)\), the coefficients are always in \(\K_T(\pt)[q_1,q_2]\). Therefore, \(\QK_T^\poly(X)\) can be identified with \(\K_T(X)\otimes\Z[q_1,q_2]\) as a module.  Since
\begin{equation*}\label{eqn:div-bun}
    \det(\cS_1)=T_1(1-\cO^{s_1}) \quad \text{and}   \quad  \det(\cS_2)=T_1T_2(1-\cO^{s_2}), 
\end{equation*}
it follows that \(\QK_T^\poly(X)\) is also generated by \(\cS_1\) and \(\det(\cS_2)\) over \(\K_T(\pt)[q_1,q_2]\), proving the surjectivity claim.


\newcommand{\etalchar}[1]{$^{#1}$}
\providecommand{\bysame}{\leavevmode\hbox to3em{\hrulefill}\thinspace}

\end{document}